\documentclass[12pt,reqno]{article}
\usepackage{amsmath,amsthm,amssymb, latexsym}
\usepackage{undertilde}

\usepackage{color}

\theoremstyle{plain}
\newtheorem{lem}{Lemma}[section]
\newtheorem{thm}[lem]{Theorem}

\theoremstyle{definition}

\newtheorem{rem}[lem]{Remark}

\newtheorem{defn}[lem]{Definition}

\numberwithin{equation}{section} \numberwithin{figure}{section}
\numberwithin{table}{section} 

\begin{document}
\title{A note on connectivity of splitting matroids}


\author{S. B. Dhotre\\ Department of Mathematics\\ Savitribai Phule Pune University, Pune-411007 (India)\\
	E-mail: dsantosh2@yahoo.co.in \\ and \\  P. P. Malavadkar\\ Department of Mathematics\\ MIT World Peace University, Pune-411038 (India)\\ E-mail: prashant.malavadkar@mitwpu.edu.in}

\date{}

\maketitle

\begin{abstract}
\noindent Fleischner introduced the idea of splitting a vertex of degree at least three in a connected graph and used the operation to characterize Eulerian graphs. Raghunathan et. al. extended the splitting operation from graphs to binary matroids. It has been studied that splitting operation, in general, may not preserve the connectedness of the binary matroid. Interestingly, it is true that the splitting matroid of a disconnected matroid may be connected.  In this paper, we characterize the binary disconnected matroids whose splitting matroid is connected.
\end{abstract}

\noindent {\bf AMS Subject Classification}: 05B35\\
{\bf Keywords}: {\it Graph; binary matroid; splitting operation; connected binary matroid}

\normalsize

\section {Introduction }
\noindent  Fleischner \cite{FS} introduced the idea of splitting a vertex of degree at least three in a connected graph and used the operation to characterize Eulerian graphs. In fact Fleischner \cite{FS} proved
the following result concerning the connectedness of a connected graph after splitting operation.

\begin{lem} {\bf (Splitting Lemma) :} Let $G$ be a connected bridgeless
	graph. Suppose $v\in V(G)$ with $d(v)>3$ and $x, y, z$ are the
	edges incident at $v$. Form the graph $G_{x, y}$ and $G_{x, z}$ by
	splitting away the pairs $x, y$ and $x, z$ respectively, and
	assume $x$ and $z$ belong to different blocks if $v$ is a cut vertex
	of $G$. Then either $G_{x, y}$ or $G_{x, z}$ is connected and
	bridgeless.\end{lem}

Raghunathan, Shikare and Waphare  \cite{RSWS}  extended the splitting operation from graphs to binary matroids. This operation is defined for a pair of elements of a binary  matroid in the following way:
\begin{defn} Let $M$ be a binary matroid on a set $E$ and $A$ be a matrix over $GF(2)$ that represents the matroid $M$. Consider elements $x$ and $y$ of
$M$. Let $A_{x,y}$ be the matrix that is obtained by adjoining an
extra row to $A$ with this row  being zero everywhere except in
the columns corresponding to $x$ and $y$ where it takes the value
$1$. Let $M_{x,y}$ be the matroid represented by  the matrix
$A_{x,y}$. We say that $M_{x,y}$ is obtained from $M$ by
splitting the pair of elements $x$ and $y$. Moreover, the
transition from $M$ to $M_{x,y}$ is called the {\it splitting
operation}. The two elements $x$ and $y$ of the matroid $M_{x,y}$
are now in series. \end{defn}

\begin{lem} \label{Lemma 1.3}\cite{RSWS}. Let $M$ be a binary  matroid and $x,y\in
	E(M).$  Then \\
	(i)~$ M_{x, y} = M $ if and only if $ x$ and $y$ are in series in $M;$\\
	(ii)~$x$ and $y$ are in series in $M_{x,y};$\\
	(iii)~if $x$ and $y$ are not in series in $M$ then, $r'(M_{x, y}) = r(M) + 1,$
	where $r$ and $r'$ are rank functions of $M$ and $ M_{x, y}$, respectively; and\\
	(iv)~for any $ X\subseteq E(M),$ $ r(X) \leq r'(X)\leq r(X) + 1.$ \end{lem}

\begin{lem}\label{Lemma 1.4}\cite{AM}.  Let $M$ be a binary matroid
	and let $x,y\in E(M).$ If $C^*$ is a cocircuit of $M$ containing both
	$x$ and $y$ with $ |C^*| \geq 3,$ then $C^{*}-\{x,y\}$ is a cocircuit of
	$M_{x,y}.$\end{lem}

Shikare, Azadi and Waphare \cite{SAWGS} further generalized this operation and defined the splitting operation for arbitrary set of elements in a binary matroid. The generalized  operation is defined in the following way.

\begin{defn}\label{defngs} Let $M = M[A]$ be a binary matroid with ground set $E$ and suppose $X$ is a subset of $E$.  Let $A_X$ be the matrix obtained from $A$  by adjoining an extra row to $A$ with this row being zero everywhere except in the columns corresponding to the elements of $X$ where it takes the value 1.  The splitting matroid $M_X$ is defined to be the vector matroid of the matrix $A_X$. The	transition from $M$ to $M_X$ is called a  {\it generalized splitting operation}.\end{defn}

Let $M$ be a matroid and $X\subset E(M)$. We assume that $M$ is loopless and coloopless. The set of circuits of $M$ is denoted by $\cal C$$(M)$. We call a circuit of $M$ as an {\it $OX$-circuit} if it contains an odd number of elements of the set $X$. Using Definition \ref{defngs}, Shikare, Azadi and Waphare \cite{SAWGS} characterized the circuits of the splitting matroid $M_X$.

\begin{lem} Let $M$ be a binary matroid on $E$ and suppose $X \subseteq E$.  Then $\cal C$$(M_X) = $ $\cal C$$_0	\cup $ $\cal C$$_1$ where\\
\begin{description}
\item $\cal C$$_0 = \{ C \in \cal C$$(M)~ |~ C$ contains an even number of elements of $X$ \}; and  \\
\item $\cal C$$_1$ = The set of minimal members of $\{C_1 \cup C_2 ~|~ C_1, C_2 \in $$\cal C$$(M), C_1 \cap C_2 = \phi$ and each of $C_1$ and $C_2$ is an $OX$-circuit such that $C_1 \cup C_2$ contains no member  of $\cal C$$_0\}$.
\end{description} \end{lem}

In the following lemma, Shikare, Azadi and Waphare \cite{SAWGS} characterized the rank function of the matroid $M_X$ in terms of the rank function of the matroid $M$.

\begin{lem} \label{rank} Let $r$ and $r'$ be the rank functions of the matroids $M$ and $M_X$, respectively. Suppose $A \subseteq E(M)$. Then
	\begin{eqnarray}
	r'(A) &=& r(A) + 1 \mbox{ if $A$ contains an $OX$-circuit of $M$; and}\\
	&=& r(A) \mbox{ if $A$ contains no $OX$-circuit of $M$. }
	\end{eqnarray}
\end{lem}


The concept of $n$-connection for matroids was introduced by W. T. Tutte  based upon the corresponding idea for graphs (see \cite{TutteCM}).   
The splitting operation, in general, does not preserve the connectedness of the binary matroid. Several results concerning splitting operation have been explored in \cite{BDCSM, AM, SSL, SABS}.

In the next result Shikare \cite{SSL} provided a sufficient condition for the splitting operation to yield a connected binary matroid from a $4$-connected binary matroid.
\begin{thm} Let $M$ be a $4$-connected binary matroid with $|E(M)|\geq 9$ and let $x, y$ be distinct elements of $M$. Then $M_{x,y}$ is  connected binary matroid. \end{thm}

Borse and Dhotre \cite{BDCSM} strengthened Shikare's result by proving that $M_{x,y}$ is connected for every $x, y \in E(M)$ whenever $M$ is connected and vertically $3$-connected with cogirth at least $4$ and girth at least $3$.

\begin{thm} Let $M$ be a connected and vertically $3$-connected binary matroid and $x, y$ be distinct elements of $M$. Suppose that every cocircuit $Q$ of $M$ containing $x, y$ is of size at least $4$ and further, $Q$ does not contain a $2$-circuit of $M$. Then $M_{x,y}$ is  connected binary matroid. \end{thm}

The following result provides a necessary condition for a matroid to be $n$-connected (see \cite{Oxley}).
\begin{lem}\label{ncnd} If $M$ is an $n$-connected matroid and $|E(M)| \geq 2(n-1)$ then all circuits and all cocircuits of $M$ have at least $n$ elements. \end{lem}

The generalized splitting operation on a connected binary matroid, in general, may not yield a connected binary matroid. If $M$ is a connected binary matroid and $|X| < 2$ then $X$ will be a cocircuit of $M_X$ of size less than $2$ (see \cite{AM}). Therefore, by Lemma \ref{ncnd}, $M_X$ is not connected. In the following Theorem, Malavdkar et.al \cite{MDSCESM} characterized $n$-connected binary matroids which yields $n$-connected binary matroids under generalized splitting operation.

\begin{thm}\label{ncsnc} Let $M$ be an $n$-connected and vertically $(n+1)$-connected binary matroid, $n \geq 2$, $|E(M)| \geq 2(n-1)$ and girth of $M$ is at least $n+1$. Let $X\subset E(M)$ with $|X| \geq n$. Then $M_X$ is $n$-connected if and only if for any $(n-1)$-element subset $S$ of $E(M)$ there is an $OX$-circuit $C$ of $M$ such that $S \cap C = \phi$.\end{thm}

\section{Disconnected Binary Matroids and the Splitting Operation}

It has been studied that the splitting matroid of a binary connected matroid need not be connected. Interestingly, it is  true that the splitting matroid of a disconnected matroid may be connected. In the following lemma, we characterize those binary disconnected matroids whose splitting matroid is connected.

\begin{lem}  Let $M$ be a binary disconnected matroid, the elements $x,y \in E(M)$  and $x$ , $y$ are not in a 2-cocircuit of $M$. Then $M_{x,y}$ is connected if and only if $M$ has exactly two components each of which contains either $x$ or $y.$ \end{lem}
	
\begin{proof} Suppose that $M_{x,y}$ is connected. Let $D$ be a	 component of $M$ such that $D \cap \{x,y\}= \phi.$ Let $z \in D$		and $C$ be a circuit in $M_{x,y}$ containing $z$ and $x$. But then $y \in C$. Therefore, either $C$ is a circuit of $M$ or it is the		union of two disjoint circuits $C_x$ and $C_y$ (where $C_x$ is a circuit containing $x$ but not $y$ and $C_y$ is a circuit containing $y$ but not $x$ and $C_x\cup C_y$ contains no circuit of $M$ containing both $x, y$ or neither or a circuit of $M$, $C_{xy}$ containing both $x, y$. Thus, we get a circuit in $M$ containing $z$ and one of the $x$ and $y$. Thus $x$ or $y$ is in $D,$ a contradiction to the fact that $D \cap \{x,y\}=\phi.$ Further, if $x \in D_1,y \in D_2$ and there exists another component $D_3$ such that $D_3 \neq D_1 \neq D_2.$ But then either $x$ or $y \in D_3.$ This means either $D_3 = D_1$ or $D_3=D_2.$ This implies that $M$ has at most two components. 

Conversely, if $x \in D_1, y \in D_2$ and $D_1,D_2$ are two components of $M$. Then, by Lemma 2.5 of \cite{SSL}, $M_{x,y}$ is connected.\end{proof}
		
\begin{lem}\label{components} Let $D_1,D_2,D_3,\cdots, D_{t-1},D_t$ be components of a matroid $M$. Then $r(D_1)+r(D_2)+r(D_3)+ \cdots + r(D_{t-1})+r(D_t)=r(M).$\end{lem}
\begin{proof} Let $B_i$ be the basis of $D_i$. Since there is no circuit containing one element of $D_i$ and one element of $D_j, i\neq j$, $~B_i \cup B_j$ is an independent set of $M$. Thus, $B=B_1 \cup B_2 \cup B_3 \cup \cdots \cup B_{t-1} \cup B_t$ is independent set in $M$. Let $e\in E(M)$. Then $B\cup e$ is not independent, as $B_k \cup e$ contains a circuit $C$ of	$M$ for some $k \in \{1,2,...,t-1,t\}$ and so $C$ is a circuit in $B\cup e$. Hence $B$ is maximal independent set. That is, $B$ is a basis of $M$.
	
Therefore, $r(M)=r(B_1 \cup B_2 \cup B_3 \cup \cdots B_{t-1} \cup B_t)$\\
			$=|B_1+ B_2+B_3+ \cdots + B_{t-1}+ B_t|$\\
			$=|B_1|+ |B_2|+ \cdots + |B_{t-1}|+ |B_t|$\\
			$=r(D_1)+r(D_2)+r(D_3) + \cdots + r(D_{t-1})+ r(D_t)$\end{proof}
		
\begin{lem} Let $M$ be a connected matroid and $x,y \in E(M)$. Suppose $x$ is not parallel to $y$. If $M \setminus \{x,y\}$ is connected, then $M_{x,y}$ is connected. \end{lem}
\begin{proof} Let $D=M \setminus \{x,y\}$. Then $D$ is connected. Now $M \setminus \{x,y\}= M_{x,y} \setminus \{x,y\}$ and $D$ is contained in a component of $M_{x,y}$. If there exists a circuit of $M$ containing $x$ and $y$, then it is	preserved in $M_{x,y}.$ Since $x$ and $y$ are not parallel, there is a circuit containing $x$ and $y$ that intersects $D$. Hence, $D$ and $x,y$ must belong to same component of $M_{x,y}$. That is $M_{x,y}$ has only one component. Thus $M_{x,y}$ is connected.\end{proof}
		
\begin{rem} If $x$ is parallel to $y$, and $M$ and $M \setminus \{x,y\}$  are connected,  $M_{x,y}$ need not be connected.\end{rem} We have given the following example to show this.
		
Consider the cycle matroid $M(G)$ of graph $G$ shown in Figure 1. We know that the matroid $M(G)$ with at least 2 elements is connected if $G$ is 2-connected and $(M(G))_{x,y}=M(G_{x,y}).$

In Figure 1, $G-\{x,y\}$ is 2-connected but $G_{x,y}$ is not 2-connected. Hence $M(G_{x,y})$ is not connected.
		
\begin{defn} Let $M$ be a matroid and $x,y \in E(M)$ then $x$ and $y$ are in series if $x$ and $y$ form a 2-cocircuit\end{defn}

\begin{rem} The converse of the above Lemma is not true. The splitting matroid $M_{x,y}$ of $M$ is connected though $M\setminus \{x,y\}$ is not connected. Consider the cycle matroid $M(G)$ of graph $G$ shown in Figure 2. The graphs corresponding to $(M(G))_{x,y}$ and $(M(G)\setminus \{x,y\}$ are $G_{x,y}$ and $G-\{x,y\}$, respectively .\end{rem}

\begin{figure}
			\unitlength 1mm 
			\linethickness{0.4pt}
			\ifx\plotpoint\undefined\newsavebox{\plotpoint}\fi 
			\begin{picture}(100.75,44.125)(25,0)
			\put(50.5,15.67){\circle*{1.33}}
			\put(88.25,15.92){\circle*{1.33}}
			\put(129.25,15.17){\circle*{1.33}}
			\put(69.17,15.33){\circle*{1.33}}
			\put(106.92,15.58){\circle*{1.33}}
			\put(147.92,14.83){\circle*{1.33}}
			\put(69.17,32){\circle*{1.33}}
			\put(106.92,32.25){\circle*{1.33}}
			\put(147.92,31.5){\circle*{1.33}}
			\put(50.5,31.67){\circle*{1.33}}
			\put(88.25,31.92){\circle*{1.33}}
			\put(129.25,31.17){\circle*{1.33}}
			\put(137.25,40.42){\circle*{1.33}}
			\put(69.17,32.33){\line(0,0){0}}
			\put(106.92,32.58){\line(0,0){0}}
			\put(147.92,31.83){\line(0,0){0}}
			\put(69.17,32.33){\line(0,-1){17}}
			\put(106.92,32.58){\line(0,-1){17}}
			\put(147.92,31.83){\line(0,-1){17}}
			\put(69.17,15.33){\line(0,0){0}}
			\put(106.92,15.58){\line(0,0){0}}
			\put(147.92,14.83){\line(0,0){0}}
			\put(69.17,15.33){\line(-1,0){18.33}}
			\put(106.92,15.58){\line(-1,0){18.33}}
			\put(147.92,14.83){\line(-1,0){18.33}}
			\put(50.83,15.33){\line(0,0){0}}
			\put(88.58,15.58){\line(0,0){0}}
			\put(129.58,14.83){\line(0,0){0}}
			\put(50.83,15.33){\line(0,1){16.67}}
			\put(88.58,15.58){\line(0,1){16.67}}
			\put(129.58,14.83){\line(0,1){16.67}}
			\put(50.83,32){\line(0,0){0}}
			\put(88.58,32.25){\line(0,0){0}}
			\put(129.58,31.5){\line(0,0){0}}
			\put(59.83,39.33){\line(0,0){0}}
			\put(60.5,39.67){\line(0,0){0}}
			\multiput(51,31.25)(.0378630705,-.0337136929){482}{\line(1,0){.0378630705}}
			\multiput(88.75,31.5)(.0378630705,-.0337136929){482}{\line(1,0){.0378630705}}
			\multiput(129.75,30.75)(.0378630705,-.0337136929){482}{\line(1,0){.0378630705}}
			\multiput(69.25,32.25)(-.037755102,-.0336734694){490}{\line(-1,0){.037755102}}
			\multiput(107,32.5)(-.037755102,-.0336734694){490}{\line(-1,0){.037755102}}
			\multiput(148,31.75)(-.037755102,-.0336734694){490}{\line(-1,0){.037755102}}
			\put(101.75,4.75){\makebox(0,0)[cc]{\bf Figure 1}}
			\qbezier(50.5,31.5)(59.625,44.125)(69.25,32.25)
			\multiput(69.25,32.25)(-.8043478,-.0326087){23}{\line(-1,0){.8043478}}
			\put(50.75,31.5){\line(0,1){0}}
			\multiput(137.25,40.5)(.0402621723,-.0337078652){267}{\line(1,0){.0402621723}}
			\qbezier(136.75,41.25)(150.75,44.125)(147.75,31.5)
			\put(59.25,40){\makebox(0,0)[cc]{$x$}}
			\put(148.75,39.75){\makebox(0,0)[cc]{$x$}}
			\put(59.5,33.25){\makebox(0,0)[cc]{$y$}}
			\put(141.5,34){\makebox(0,0)[cc]{$y$}}
			\put(59.25,11){\makebox(0,0)[cc]{$G$}}
			\put(95.25,11.5){\makebox(0,0)[cc]{$G-\{x,y\}$}}
			\put(138.5,11.5){\makebox(0,0)[cc]{$G_{x,y}$}}
			\end{picture}
		\end{figure}
\begin{figure}
			\unitlength 1mm 
			\linethickness{0.4pt}
			\ifx\plotpoint\undefined\newsavebox{\plotpoint}\fi 
			\begin{picture}(148.585,39.67)(25,0)
			\put(50.5,15.67){\circle*{1.33}}
			\put(88.25,15.92){\circle*{1.33}}
			\put(129.25,15.17){\circle*{1.33}}
			\put(69.17,15.33){\circle*{1.33}}
			\put(106.92,15.58){\circle*{1.33}}
			\put(147.92,14.83){\circle*{1.33}}
			\put(69.17,32){\circle*{1.33}}
			\put(106.92,32.25){\circle*{1.33}}
			\put(147.92,31.5){\circle*{1.33}}
			\put(50.5,31.67){\circle*{1.33}}
			\put(88.25,31.92){\circle*{1.33}}
			\put(129.25,31.17){\circle*{1.33}}
			\put(69.17,32.33){\line(0,0){0}}
			\put(106.92,32.58){\line(0,0){0}}
			\put(147.92,31.83){\line(0,0){0}}
			\put(69.17,32.33){\line(0,-1){17}}
			\put(106.92,32.58){\line(0,-1){17}}
			\put(147.92,31.83){\line(0,-1){17}}
			\put(69.17,15.33){\line(0,0){0}}
			\put(106.92,15.58){\line(0,0){0}}
			\put(147.92,14.83){\line(0,0){0}}
			\put(69.17,15.33){\line(-1,0){18.33}}
			\put(147.92,14.83){\line(-1,0){18.33}}
			\put(50.83,15.33){\line(0,0){0}}
			\put(88.58,15.58){\line(0,0){0}}
			\put(129.58,14.83){\line(0,0){0}}
			\put(50.83,15.33){\line(0,1){16.67}}
			\put(88.58,15.58){\line(0,1){16.67}}
			\put(129.58,14.83){\line(0,1){16.67}}
			\put(50.83,32){\line(0,0){0}}
			\put(88.58,32.25){\line(0,0){0}}
			\put(129.58,31.5){\line(0,0){0}}
			\put(59.83,39.33){\line(0,0){0}}
			\put(60.5,39.67){\line(0,0){0}}
			\put(101.75,4.75){\makebox(0,0)[cc]{\bf Figure 2}}
			\multiput(69.25,32.25)(-.8043478,-.0326087){23}{\line(-1,0){.8043478}}
			\put(50.75,31.5){\line(0,1){0}}
			\put(59.25,11){\makebox(0,0)[cc]{$G$}}
			\put(95.25,11.5){\makebox(0,0)[cc]{$G-\{x,y\}$}}
			\put(138.5,11.5){\makebox(0,0)[cc]{$G_{x,y}$}}
			\put(59.75,34.5){\makebox(0,0)[cc]{$x$}}
			\put(63,13.25){\makebox(0,0)[cc]{}}
			\put(61.75,12.5){\makebox(0,0)[cc]{}}
			\put(59.5,18){\makebox(0,0)[cc]{$y$}}
			\multiput(129.25,31.25)(2.375,.03125){8}{\line(1,0){2.375}}
			\put(139,35.75){\makebox(0,0)[cc]{$x$}}
			\put(126.25,23){\makebox(0,0)[cc]{$y$}}
			\end{picture}
\end{figure}
		
In Figure 2, $G-\{x,y\}$ is disconnected while $G_{x,y}$ is 2-connected. The graphs $G$ and $G_{x,y}$ are isomorphic follows from the facts that $(M(G))_{x,y}=M(G_{x,y})$ and $x,y$ are in series in $M(G).$

		\centerline{************}

\bibliographystyle{amsplain}

\end{document}